\documentclass[11pt]{article}

\usepackage{amsthm,amsmath,amssymb,amsbsy,amsfonts}
\usepackage[dvipsnames,usenames]{color}
\usepackage{fullpage}
\usepackage{graphicx}
\usepackage{subfigure}
\numberwithin{equation}{section}

\title{\bf\Large Optimal pointwise estimates for derivatives of solutions to Laplace, Lam\'e
and Stokes equations}

{ \date\ }

\newtheorem{lemma}{Lemma}
\newtheorem{theorem}{Theorem}

\newtheorem{proposition}[theorem]{Proposition}
\newtheorem{corollary}{Corollary}

\newcommand{\bs}{\boldsymbol}
\newenvironment{remark}{{\bf Remark}}

\begin{document}
\maketitle
\vspace{-2cm}
\centerline{\scshape Gershon Kresin}
\medskip
{\footnotesize
 \centerline{Department of Computer Science and Mathematics}
   \centerline{Ariel University}
   \centerline{44837 Ariel, Israel}
   \centerline{kresin@ariel.ac.il}} 
\medskip
\centerline{\scshape Vladimir Maz'ya}
\medskip
{\footnotesize
 \centerline{Department of Mathematical Sciences}
   \centerline{University of Liverpool, M$\&$O Building}
   \centerline{Liverpool, L69 3BX, UK}
\medskip
 \centerline{Department of Mathematics}
   \centerline{Link\"oping University}
   \centerline{SE-58183 Link\"oping, Sweden}
   \centerline{vladimir.mazya@liu.se}}
\bigskip

\bigskip   
{\bf Abstract.} Various optimal estimates for solutions of the Laplace, 
Lam\'e and Stokes equations in multidimensional domains, as well as new real-part 
theorems for analytic functions are obtained. 
 
\setcounter{section}{0} 
\setcounter{equation}{0}
\section{Introduction}

In the present paper we extend our study of the best constants in certain inequalities for
solutions of the Laplace, Lam\'e and Stokes equations (see Kresin and Maz'ya \cite{KM15}).
We also deal with optimal estimates for analytic functions in the spirit of
our recent article  \cite{KM14}. Let us formulate some results obtained in the sequel.

By $|\cdot|$ we denote the Euclidean length of a vector or absolute 
value of a scalar quantity. Let $\Omega$ be a domain in ${\mathbb R}^n$.
By $d_x$ we mean the distance from a point $x\in \Omega$ to $\partial \Omega$ 
and by $\omega_n$ we denote the area of the $(n-1)$-dimensional unit sphere.

One of the results derived in Section 2 is the following pointwise estimate 
of the gradient of a bounded harmonic function in the complement $\Omega$ of a convex
closed domain in ${\mathbb R}^n$: {\it
$$
\big |\nabla u(x)\big |\leq {C_n \over d_x}\;\sup_\Omega |u|  
$$
for all $x\in \Omega$. Here
$$
C_n=\frac{4(n-1)^{(n-1)/2}\;\omega_{n-1}}{ n^{n/2}\;\omega_n}
$$
is the best constant.}

\smallskip
We also state here a limit estimate with the same best constant $C_n$, valid  for 
arbitrary domains which is established in Section 2. 

{\it Let $\Omega$ be a domain in ${\mathbb R}^n$,
and let ${\mathfrak U}(\Omega)$ be the set of harmonic functions $u$ in $\Omega$
with $\sup_\Omega |u|\leq 1$.
Suppose that a point $\xi \in \partial\Omega$ can be touched by an interior ball $B$. Then
$$
\limsup _{x\rightarrow \xi}\sup_{u\in {\mathfrak U}(\Omega)}
|x-\xi|\big |\nabla u(x)\big | \leq C_n\;, 
$$
where $x$ is a point of the radius of $B$ directed from the center to $\xi$.}

\smallskip
In Section 3 we obtain pointwise estimates 
for the directional derivative $(\bs \ell, \nabla ){\bs u}$,
where $\bs u(x)=(u_{_1}(x),\dots,u_m(x))$ is a vector field whose components are
harmonic in $\Omega$. Assertions proved here are generalizations of the theorems given
in Section 2. In Section 4 we present analogs of the theorems of Section 2
containing pointwise and limit estimates for $|{\rm div}\;\bs u(x)|$, $m=n$.

\smallskip 
By $[{\rm C}_{\rm b}({\overline\Omega})]^n$ 
we mean the space of vector-valued functions
with $n$ components which are bounded and continuous on ${\overline\Omega}$.
This space is endowed with the norm $||\bs u||_{[{\rm C}_{\rm b}({\overline\Omega})]^n}
=\sup \{|\bs u(x)|: x \in {\overline\Omega}) \}$. By $[{\rm C}^{2}(\Omega)]^n$
we denote the space of $n$-component vector-valued functions with continuous
derivatives up to second order in $\Omega$. 

\smallskip 
Next, in Section 5 we find an optimal estimate for $|{\rm div}\;\bs u(x)|$, where $\bs u$
is an elastic displacement vector in ${\mathbb R}^n_+=\{ x=(x', x_n):
x'\in {\mathbb R}^{n-1}, x_n>0 \}$. 
As a corollary, we obtain an optimal estimate for the pressure $p$ in a viscous 
incompressible fluid in ${\mathbb R}^n_+$. We formulate two statements following 
from these results.

\smallskip
{\it Let $\Omega={\mathbb R}^n\backslash {\overline G}$, 
where $G$ is a convex domain in ${\mathbb R}^n$. 

{\rm (i)} Let $\bs u\in [{\rm C}^{2}(\Omega)]^n 
\cap [{\rm C}_{\rm b}(\overline {\Omega})]^n$ be a solution of the Lam\'e system
$$
\Delta \bs u +(1-2\sigma)^{-1}\;{\rm grad}\;{\rm div}\; \bs u=\bs 0
$$ 
in $\Omega$, where $ \sigma \in (-\infty, 1/2)\cup(1, +\infty)$ is the Poisson coefficient.
Then for any point $x\in \Omega$ the inequality
$$
\left |{\rm div}\;\bs u (x)\right |\leq {(1-2\sigma)E_n \over (3-4\sigma)d_x}\;
\sup_\Omega |\bs u|  
$$
holds, where
$$
E_n={4 \omega_{n-1} \over \omega_n}\int_0^{\pi/2}
\big [1+n(n-2)\cos^2\vartheta \big ]^{1/2}\sin^{n-2}\vartheta d\vartheta
$$
is the best constant.}

{\rm (ii)} {\it Let $\bs u\in [{\rm C}^{2}(\Omega)]^n \cap [{\rm C}_{\rm b}(\overline {\Omega})]^n$ be a vector component of the solution $\{ \bs u, p \}$ to the Stokes system 
$$
\Delta \bs u -{\rm grad}\; p=\bs 0\;,\;\;\;\;{\rm div}\;\bs u=0\;\;\mbox{in}
\;\;\Omega
$$
and let $p(x)$ be the pressure vanishing as $d_x\rightarrow\infty$. Then for 
any point $x\in \Omega$ the inequality
$$
|p(x)|\leq {E_n \over d_x}\;\sup_\Omega |\bs u|  
$$
holds with the same best constant $E_n$ as above.}

\smallskip
The last Section 6 is dedicated to some new real-part theorems for analytic functions
(see Kresin and Maz'ya \cite{KM} and the bibliography collected there). We derive
the following results.

\smallskip
{\rm (i)} {\it Let 
$\Omega={\mathbb C}\backslash {\overline G}$, where $G$ is a convex domain in ${\mathbb C}$,
and let $f$ be a holomorphic function in $\Omega$ with bounded real part.
Then for any point $z\in \Omega$ the inequality
$$
\big |f^{(s)}(z)\big |\leq {K_s \over d_z^s}\sup_\Omega |\Re f|\;,\;\;\;\;\;\;\;\;s=1,2,\dots\;, 
$$
holds with $d_z={\rm dist}\;(z, \partial \Omega)$, where
$$
K_s={s! \over \pi}\max_{\alpha}\int_{-\pi/2}^{\pi/2} \big |\cos \big (\alpha +(s+1)
\varphi \big ) \big |\cos ^{s-1}\varphi\;d\varphi
$$
is the best constant. In particular $K_{2l+1}=2[(2l+1)!!]^2[\pi(2l+1]^{-1}$.} 

\smallskip
{\rm (ii)} {\it Let $\Omega$ be a domain in ${\mathbb C}$,
and let ${\mathfrak R}(\Omega)$ be the set of holomorphic functions $f$ in $\Omega$
with $\sup_\Omega|\Re f|\leq 1$.
Assume that a point $\zeta \in \partial\Omega$ can be touched by an interior disk $D$. Then
$$
\limsup _{z\rightarrow \zeta} \sup_{f\in {\mathfrak R}(\Omega)}|z-\zeta|^s\big |f^{(s)}(z)\big | \leq K_s\;,\;\;\;\;\;\;\;\;s=1,2,\dots\;, 
$$
where $z$ is a point of the radius of $D$ directed from the center to $\zeta$.
Here the constant $K_s$ is the same as above and cannot be diminished.}

\smallskip
More details concerning the above formulations can be found in the statements of
corresponding theorems, propositions and corollaries in what follows.

\setcounter{equation}{0}
\section{Estimates for the gradient of harmonic function}

We introduce some notation used henceforth. 
Let ${\mathbb B} =\big \lbrace x \in {\mathbb R}^{n}: |x|<1 \big \rbrace$,
${\mathbb B}_R =\big \lbrace x \in {\mathbb R}^{n}: |x|<R \big \rbrace$, 
and ${\mathbb S}^{n-1}=\{ x \in {\mathbb R}^{n}: |x|=1 \}$.  
By $h^\infty(\Omega)$ we denote  the Hardy space of bounded harmonic functions
on the domain $\Omega$ with the norm $||u||_{h^\infty(\Omega)} =\sup \{ | u(x) |: 
x\in \Omega \}$. 

\begin{theorem} \label{T_0} 
Let $\Omega={\mathbb R}^n\backslash {\overline G}$, where $G$ is a 
convex domain in ${\mathbb R}^n$, and let $u$ be a bounded harmonic function 
in $\Omega$. Then for any point $x\in \Omega$ the inequality
\begin{equation} \label{Eq_1.AG}
\big |\nabla u(x)\big |\leq {C_n \over d_x}\;\sup_\Omega |u|  
\end{equation}
holds, where
\begin{equation} \label{Eq_1.AB}
C_n=\frac{4(n-1)^{(n-1)/2}\;\omega_{n-1}}{ n^{n/2}\;\omega_n}
\end{equation}
is the best constant in the inequality
$$
|\nabla u(x)|\leq C_n\;x_n^{-1}\;
||u||_{L^\infty(\partial{\mathbb R}^{n} _{+})} 
$$
for a bounded harmonic function $u$ in the half-space ${\mathbb R}^{n} _{+}$.

In particular, 
$$
C_2={2\over \pi}\;,\;\;\;\;\;\;\;\;\;\;\;\;\;C_3={4\over 3\sqrt{3}}\;.
$$
\end{theorem}
\begin{proof} Let $\xi \in \partial\Omega$ be a point at $\partial\Omega$ nearest to
$x \in \Omega$ and let $T(\xi)$ be the hyperplane containing $\xi$ and orthogonal 
to the line joining $x$ and $\xi$. By ${\mathbb R}^n_\xi$ we denote the open half-space 
with boundary $T(\xi)$ such that ${\mathbb R}^n_\xi\subset\Omega $.

Let $n\geq 3$. According to Theorem 1 \cite{KM12}, 
the inequality
\begin{equation} \label{Eq_1.ABC}
|\nabla u(x)|\leq {C_n\over d_x}\;||u||_{h^\infty({\mathbb R}^n_\xi)}
\end{equation}
holds, where $C_n$ is given by (\ref{Eq_1.AB}). Using (\ref{Eq_1.ABC}) and the 
obvious inequality
$$
||u||_{h^\infty({\mathbb R}^n_\xi)}\leq \sup_\Omega |u|\;,
$$
we arrive at (\ref{Eq_1.AG}).

The case $n=2$ is considered analogously, the role of (\ref{Eq_1.ABC}) being 
played by the estimate 
\begin{equation} \label{Eq_HP}
|f'(z)|\leq {2 \over \pi \Im z}\sup_{{\mathbb C}_+} |\Re f|
\end{equation}
(see \cite{KM}, Sect. 3.7.3)
by the change $f=u+iv$, $f'(z)=u'_x-iv'_y$, where $f$ is a holomorphic function in 
${\mathbb C}_+=\{z\in {\mathbb C}: \Im z>0 \}$ with bounded real part.
\end{proof}

\smallskip
In what follows, we assume that the Cartesian coordinates with origin 
${\mathcal O}$ at the center of the ball are chosen in such a way that 
$x=|x|\bs e_n$. By $\bs\ell$ we denote an arbitrary unit vector in ${\mathbb R}^{n}$ 
and by $\bs \nu_{\!x}$ we mean the unit vector of exterior normal to the sphere 
$|x|=r$ at a point $x$. Let $\bs\ell_\tau$ be the orthogonal 
projection of $\bs\ell$ on the tangent hyperplane to the sphere $|x|=r$ at $x$. 
If $\bs\ell_\tau\neq \bs 0$, we set $\bs \tau_{\!\!x}=\bs\ell_\tau /|\bs\ell_\tau |$, 
otherwise $\bs \tau_{\!\!x}$ is an arbitrary unit vector tangent to the
sphere $|x|=r$ at $x$. Hence 
\begin{equation} \label{Eq_2.1.1S}
\bs \ell=\ell _{\tau}\bs \tau_{\!\!x}+\ell _{\nu}\bs \nu_{\!x},
\end{equation} 
where $\ell _{\tau}=|\bs\ell_\tau |$ and $\ell _{\nu}=(\bs\ell, \bs \nu_{\!x})$. 

We premise Lemmas 1 and 2 to Theorem 2. 
In Lemma 1 we derive a representation for the sharp coefficient ${\mathcal K}_n(x)$
in the inequality
\begin{equation} \label{Eq_2.1.1}
|\nabla u(x)|\leq {\cal K}_n( x)||u||_{L^\infty({\partial\mathbb B})}\;,
\end{equation}
where $x\in {\mathbb B}$ and $u\in h^\infty({\mathbb B})$. {\sl Here and elsewhere 
we say that a certain coefficient is sharp if it cannot be diminished for any 
point $x$ in the domain under consideration}.
The expression for ${\mathcal K}_n( x)$, 
given below, contains two factors one of which is an explicitely given function
increasing to infinity as $r \rightarrow 1$ and the second factor (the double integral) 
is a bounded function on the interval $0\leq r\leq 1$.

\setcounter{theorem}{0}
\begin{lemma} \label{P_1}
Let $u \in h^\infty({\mathbb B})$, and let $x $ be an arbitrary point in ${\mathbb B}$. The sharp 
coefficient ${\cal K}_n(x)$ in inequality $(\ref{Eq_2.1.1})$ is given by
\begin{equation} \label{Eq_2.1.2}
{\mathcal K}_n(x)={ 2^{n-2}(n-2) \over \pi(1+r)^{n-1}(1-r)}
\sup _{\gamma \geq 0}{1\over \sqrt{1+\gamma^2}} \int_0^{\pi}\sin ^{n-3}\varphi  \;d\varphi
\int_0^{\pi/2} G_{n}(\vartheta, \varphi ; r, \gamma)\;d\vartheta \;,
\end{equation} 
where 
\begin{equation} \label{Eq_2.1.3}
G_{n}(\vartheta, \varphi ; r, \gamma)={\big | n\cos 2 \vartheta+
n\gamma\sin 2\vartheta\cos \varphi +(n-2)r \big | 
\over \left [1+\left ({1-r \over 1+r} \right )^2 \tan^2 \vartheta \right ]^{(n-2)/2}}
\sin^{n-2}\vartheta\;.
\end{equation} 
\end{lemma}
\begin{proof}  {\it 1. Representation for ${\cal K}_n( x)$ by an 
integral over ${\mathbb S}^{n-1}$}. Let $u$ stand for a harmonic function in ${\mathbb B}$
from the space $h^\infty({\mathbb B})$. By Poisson formula we have  
\begin{equation} \label{Eq_2.2} 
u(x)={1 \over \omega _n}\int _{{\mathbb S}^{n-1}}{1-r^2 \over |y-x|^n}u(y)d\sigma_y\;.
\end{equation}

Fix a point $x \in {\mathbb B}$. By (\ref{Eq_2.2})
$$
{\partial u \over \partial x_i}={1 \over \omega _n }\int _{{\mathbb S}^{n-1}}
\left [ {-2x_{i} \over |y-x|^n} + 
{n \left ( 1-r^2 \right ) (y_i-x_i)\over |y-x|^{n+2}}\right ]u(y)d\sigma _y,
$$
that is
$$
\nabla u(x)={1 \over \omega _n}\int _{{\mathbb S}^{n-1}}
{n\left ( 1-r^2 \right )(y-x)-2|y-x|^2 x \over |y-x|^{n+2}}\;u(y)d\sigma _y.
$$
Thus 
$$
(\nabla u(x), \bs\ell )={1 \over \omega _n}\int _{{\mathbb S}^{n-1}}
{(n\left ( 1-r^2 \right )(y-x)-2|y-x|^2 x, \bs\ell ) \over |y-x|^{n+2}}\;u(y)d\sigma _y\;,
$$
and therefore
\begin{equation} \label{EQ_2A}
{\cal K}_n( x)={1 \over \omega _n}\sup_{|\bs\ell|=1}\int _{{\mathbb S}^{n-1}}
{\big |(n\left ( 1-r^2 \right )(y-x)-2|y-x|^2 x, \bs\ell )\big | \over |y-x|^{n+2}}\;d\sigma _y\;.
\end{equation}
Using (\ref{Eq_2.1.1S}), we obtain
$$
{\cal K}_n( x)={1 \over \omega _n}\sup_{|\bs\ell|=1}\int _{{\mathbb S}^{n-1}}
{\big |(n(1-r^2)\big ((y, \bs \nu_x)-r  )-2r|y-x|^2\big )\ell_\nu+ n(1-r^2)
(y, \bs \tau_x) \ell_\tau \big | \over |y-x|^{n+2}}\;d\sigma _y\;.
$$
The last expression can be written as
\begin{equation} \label{EQ_1}
{\cal K}_n( x)={a_n(r) \over \omega _n}\sup_{|\bs\ell|=1}\int _{{\mathbb S}^{n-1}}
{\big |b_n(r)(y, \bs \tau_x) \ell_\tau+\big (y_n-c_n(r) \big )
\ell_\nu \big | \over \big (1-2ry_n+r^2\big )^{(n+2)/2}}\;d\sigma _y\;,
\end{equation}
where 
\begin{equation} \label{EQ_2}
a_n(r)=n(1-r^2)+4r^2\;,\;\;\;\;\;b_n(r)={n(1-r^2) \over n(1-r^2)+4r^2}\;,\;\;\;\;\;
c_n(r)={n(1-r^2)+2(1+r^2) \over n(1-r^2)+4r^2}r\;.
\end{equation}

{\it 2. Representation for ${\cal K}_n( x)$ by a double integral}. 
Introducing the function
\begin{equation} \label{Kern}
{\cal H}_n(s, t; r, \bs\ell)={\left | b_n(r)s\ell _{\tau}+
\big ( t-c_n(r)\big )\ell _{\nu} \right | \over 
\big ( 1-2r t +r^2 \big )^{(n+2)/2}}\;,
\end{equation}
we write the integral in (\ref{EQ_1}) as the sum
\begin{eqnarray} \label{E_DEF1}
\int _ {{\mathbb S}^{n-1}_+}{\cal H}_n
((y_\tau, \bs \tau_{\!\!x}), y_n; r, \bs\ell)\;d\sigma_y +\int _ {{\mathbb S}^{n-1}_-}{\cal H}_n
((y_\tau, \bs \tau_{\!\!x}), y_n; r, \bs\ell)\;d\sigma_y ,
\end{eqnarray}
where ${\mathbb S}^{n-1}_+=\{y \in {\mathbb S}^{n-1}: (y, \bs e_n) >0 \}$, 
${\mathbb S}^{n-1}_-=\{y \in {\mathbb S}^{n-1}: (y, \bs e_n) <0 \}$. 

Let $y'=(y_1,\dots,y_{n-1}) \in {\mathbb B}'=\{y' \in {\mathbb R}^{n-1}: |y'|<1 \}$.
We put
$$
\bs \tau_{\!\!x}'=\sum_{i=1}^{n-1}(\bs \tau_{\!\!x}, \bs e_i)\bs e_i.
$$

Since $y_n =\sqrt{1-|y '|^2}$ for $y\in {\mathbb S}^{n-1}_+$ 
and $y_n =-\sqrt{1-|y '|^2}$ for $y \in {\mathbb S}^{n-1}_-$ and since
$d\sigma_y=dy'/\sqrt{1-|y'|^2}$, it follows that each of integrals in (\ref{E_DEF1}) 
can be written in the form
\begin{equation} \label{E_DEFP1}
\int _ {{\mathbb S}^{n-1}_+}{\cal H}_n
((y_\tau, \bs \tau_{x}), y_n; r, \bs\ell)d\sigma_y 
\!=\!\int _ {{\mathbb B}^{'}}{{\mathcal H}_n\Big  (  
(y',  \bs \tau_{x}'), \sqrt{1-|y'|^2}; r, \bs\ell \Big  )\over \sqrt{1-|y'|^2}}dy',  
\end{equation}
\begin{equation} \label{E_DEFP2}
\int _ {{\mathbb S}^{n-1}_-}{\cal H}_n
((y_\tau, \bs \tau_{x}), y_n; r, \bs\ell)d\sigma_y 
=\int _ {{\mathbb B}^{'}}{{\mathcal H}_n \Big  (  
(y',  \bs \tau_{x}'), -\sqrt{1-|y'|^2}; r, \bs\ell \Big  )\over \sqrt{1-|y'|^2}}dy'.  
\end{equation}

Putting
\begin{equation} \label{E_DEFIN}
{\cal M}_n(s, t; r, \bs\ell)={\cal H}_n(s, t; r, \bs\ell)+
{\cal H}_n(s, -t; r, \bs\ell),
\end{equation}
and using (\ref{Kern})-(\ref{E_DEFP2}), we rewrite (\ref{EQ_1}) as
\begin{equation} \label{E_DEFPM}
{\cal K}_n( x)={a_n(r) \over \omega _n}\sup_{|\bs\ell|=1}
\int _ {{\mathbb B}^{'}}\!\!{{\mathcal M}_n
\Big ((y',  \bs \tau_{\!\!x}'), \sqrt{1-|y'|^2}; r, \bs\ell \Big  )
\over \sqrt{1-|y'|^2}}\;dy'\;.  
\end{equation}

By the identity
$$
\int_{{\mathbb B}^n}g\big ((\bs y, \bs \xi),\;|\bs y|\big )dy=\omega_{n-1}
\int_0^1 \rho^{n-1}d\rho
\int_0^\pi g\big (|\bs \xi|\rho \cos \varphi ,\; \rho\big )\sin ^{n-2}\varphi \;d\varphi 
$$
(see, e.g., \cite{PBM}, \textbf{3.3.2(3)}),
we transform the integral in (\ref{E_DEFPM}):
\begin{eqnarray} \label{E_DEFA}
& &\hspace{-10mm}\int _ {{\mathbb B}^{'}}\!\!{{\mathcal M}_n\Big  (  
(y',  \bs \tau_{\!\!x}'), \sqrt{1-|y'|^2}; r, \bs\ell \Big  )
 \over \sqrt{1-|y'|^2}}\;dy' \\
& &\nonumber\\ 
& &=\omega_{n-2}\int_0^1 \!\!{\rho^{n-2}\over \sqrt{1-\rho^2}}\;d\rho\!
\int_0^\pi {\mathcal M}_n\Big  (\rho\cos  \varphi ,\; \sqrt{1-\rho^2}; 
r, \bs\ell \Big  )\sin ^{n-3}\varphi d\varphi \nonumber.
\end{eqnarray}
The change $\rho=\sin \theta $ in (\ref{E_DEFA}) gives
\begin{eqnarray} \label{E_DEFH}
& &\hspace{-9mm}\int _ {{\mathbb B}^{'}}\!\!{{\mathcal M}_n\Big  (  
(y',  \bs \tau_{\!\!x}'), \sqrt{1-|y'|^2}; r, \bs\ell \Big  )
 \over \sqrt{1-|y'|^2}}\;dy' \\
& &\nonumber\\ 
& &=\omega_{n-2} \int_0^{\pi/2}\sin^{n-2}\theta d\theta 
\int_0^\pi \!\!\!{\mathcal M}_n\Big  (\sin \theta\cos  \varphi ,\; \cos \theta; 
r, \bs\ell \Big  )
\sin ^{n-3}\varphi  \;d\varphi\;.\nonumber
\end{eqnarray}
Applying (\ref{Kern}), (\ref{E_DEFIN}) and introducing the notation
\begin{eqnarray*}
\hspace{-15mm}{\mathcal F}_n( \theta, \varphi; r, \bs\ell)&=&{\mathcal H}_n
\Big  (\sin \theta\cos  \varphi ,\; \cos \theta;r, \bs\ell \Big  )\nonumber\\
\hspace{-15mm}& &\nonumber\\
\hspace{-15mm}&=&
{\left |b_n(r)\ell_\tau\sin \theta \cos \varphi+\big ( \cos \theta-c_n(r)\big )
\ell_\nu \right | \over 
\big ( 1-2r\cos \theta +r^2 \big )^{(n+2)/2}}\;,
\end{eqnarray*}
we write (\ref{E_DEFH}) as follows
\begin{eqnarray} \label{E_DEFV}
& &\hspace{-5mm}\int _ {{\mathbb B}^{'}}\!\!{{\mathcal M}_n\Big  (  
(y',  \bs \tau_{\!\!x}'), \sqrt{1-|y'|^2}; r, \bs\ell \Big  )
\over \sqrt{1-|y'|^2}}\;dy'\\
& &\nonumber\\
& &=\omega_{n-2}\int_0^{\pi/2}\!\!\!\!\sin^{n-2}\theta d\theta\!\!\int_0^\pi \!\!\!
\Big ({\mathcal F}_n ( \theta, \varphi; r, \bs \ell)+
{\mathcal F}_n (\pi - \theta, \varphi; r, \bs \ell) \Big )
\sin ^{n-3}\varphi  \;d\varphi\; .\nonumber
\end{eqnarray}
Changing the variable $\psi=\pi-\theta$, we obtain
\begin{eqnarray*} 
& &\int_0^{\pi/2}\sin^{n-2}\theta\; d\theta \int_0^\pi 
{\mathcal F}_n\big (\pi- \theta, \varphi; r, \bs \ell \big ) 
\sin ^{n-3}\varphi \;d\varphi\\
& &\\
& &=\int_{\pi/2}^{\pi}\sin^{n-2}\psi\; d\psi\int_0^\pi {\mathcal F}_n
\big ( \psi, \varphi; r, \bs \ell\big )  \sin ^{n-3}\varphi  \;d\varphi,
\end{eqnarray*}
which together with  (\ref{E_DEFV}) leads to the representation of (\ref{E_DEFPM}):
\begin{equation} \label{E_F5AA}
{\cal K}_n(x)={a_n(r)\omega _{n-2} \over \omega _n}\sup_{|\bs\ell|=1}
\int_0^{\pi}\sin^{n-2}\theta\;d\theta \!
\int_0^\pi \!\!{\mathcal F}_n( \theta, \varphi; r, \bs \ell)  
\sin ^{n-3}\varphi  \;d\varphi\;.
\end{equation}

{\it 3. Transformation of representation for ${\cal K}_n(x)$}. We make the change of variable
$$
\theta=2\arctan \left ( {1-r \over 1+r}\tan \vartheta \right )
$$
in (\ref{E_F5AA}). Then
\begin{equation} \label{S_1}
\sin \theta={2\left ({ 1-r \over 1+r}\right )\tan \vartheta \over 1+
\left ({ 1-r \over 1+r}\right )^2\tan^2 \vartheta}\;,
\end{equation}
\begin{equation} \label{S_2}
d\theta={2( 1-r) \over ( 1+r)\cos^2 \vartheta 
\left ( 1+\left ({ 1-r \over 1+r}\right )^2\tan^2 \vartheta \right )}\;d\vartheta\;,
\end{equation}
\begin{equation} \label{S_3}
1-2r\cos \theta +r^2={( 1-r)^2 \over \cos^2 \vartheta 
\left ( 1+\left ({ 1-r \over 1+r}\right )^2\tan^2 \vartheta \right )}\;,
\end{equation}
\begin{equation} \label{S_4}
b_n(r)\ell_\tau\sin \theta \cos \varphi+\big ( \cos \theta-c_n(r)\big )\ell_\nu=
{( 1-r)^2 \big [ 
n\ell_\tau \sin 2\vartheta\cos \varphi+\big ( n\cos 2 \vartheta+(n-2)r\big )\ell_\nu  \big ]
\over \big [n(1-r^2)+4r^2 \big ]\cos^2 \vartheta 
\left ( 1+\left ({ 1-r \over 1+r}\right )^2\tan^2 \vartheta \right )}\;.
\end{equation}
Substituting (\ref{S_1})-(\ref{S_4}) 
in (\ref{E_F5AA}), we arrive at 
\begin{equation} \label{Eq_2.1.4}
{\mathcal K}_n(x)={ 2^{n-2}(n-2) \over \pi(1+r)^{n-1}(1-r)}
\sup _{|\bs\ell|=1} \int_0^{\pi}\sin ^{n-3}\varphi  \;d\varphi
\int_0^{\pi/2} {\mathcal G}_{n}(\vartheta, \varphi ; r, \bs\ell)\;d\vartheta \;,
\end{equation} 
where 
$$
{\mathcal G}_{n}(\vartheta, \varphi ; r, \bs\ell)={\big | 
n\ell_\tau \sin 2\vartheta\cos \varphi+\big ( n\cos 2 \vartheta+(n-2)r\big )\ell_\nu  \big | 
\over \left [1+\left ({1-r \over 1+r} \right )^2 \tan^2 \vartheta \right ]^{(n-2)/2}}
\sin^{n-2}\vartheta\;.
$$

Since the integrand in (\ref{EQ_2A}) does not change when the unit vector
$\bs\ell$ is replaced by $-\bs\ell$, we may assume that 
$\ell_\nu=(\bs\ell, \bs \nu_x) > 0$ in (\ref{Eq_2.1.4}).
Introducing the parameter $\gamma =\ell_{\!\tau}/\ell_{\!\nu}$ in 
(\ref{Eq_2.1.4}) and using the equality $\ell_\tau^2+\ell_\nu^2=1$, 
we arrive at (\ref{Eq_2.1.2}) with $G_{n}(\vartheta, \varphi ; r, \gamma)$ given by (\ref{Eq_2.1.3}).
\end{proof}

By dilation, we obtain the following result, equivalent to Lemma \ref{P_1}
and involving the ball ${\mathbb B}_R$ with an arbitrary $R$.

\begin{lemma} \label{P_2}
Let $u \in h^\infty({\mathbb B}_R)$, and let $x $ be an arbitrary point in ${\mathbb B}_R$. The sharp 
coefficient ${\cal K}_{n, R}(x)$ in the inequality
$$
|\nabla u(x)|\leq {\cal K}_{n, R}( x)||u||_{L^\infty(\partial{\mathbb B}_R)}
$$
is given by
$$
{\mathcal K}_{n, R}(x)={ 2^{n-2}(n-2)R^{n-1} \over \pi(R+|x|)^{n-1}(R-|x|)}
\sup _{\gamma \geq 0}{1\over \sqrt{1+\gamma^2}} \int_0^{\pi}\sin ^{n-3}\varphi  \;d\varphi
\int_0^{\pi/2} G_{n}\left ( \vartheta, \varphi ; {|x|\over R}, \gamma \right )\;d\vartheta \;,
$$
where 
$$
G_{n}(\vartheta, \varphi ; r, \gamma)={\big | n\gamma\sin 2\vartheta\cos \varphi +n\cos 2 \vartheta+(n-2)r \big | 
\over \left [1+\left ({1-r \over 1+r} \right )^2 \tan^2 \vartheta \right ]^{(n-2)/2}}
\sin^{n-2}\vartheta\;.
$$
\end{lemma}

Now, we prove a limit estimate for the gradient of a bounded harmonic function.

\setcounter{theorem}{1}
\begin{theorem} \label{T_1}
Let $\Omega$ be a domain in ${\mathbb R}^n$,
and let $\;{\mathfrak U}(\Omega)$ be the set of harmonic functions $u$ in $\Omega$
with $\sup_\Omega |u|\leq 1$.
Assume that a point $\xi \in \partial\Omega$ can be touched by an interior ball $B$.  
Then
\begin{equation} \label{Eq_lim}
\limsup _{x\rightarrow \xi}\sup_{u\in {\mathfrak U}(\Omega)}
|x-\xi|\big |\nabla u(x)\big | \leq C_n\;, 
\end{equation}
where $x$ is a point at the radius of $B$ directed from the center to $\xi$.
Here the constant $C_n$ is the same as in Theorem $1$.
\end{theorem}
\begin{proof}
Let $n\geq 3$. By Lemma \ref{P_2}, the relations
\begin{equation} \label{Eq_2.1.8}
\limsup_{|x|\rightarrow R}\;\sup \big \{ (R-|x|)|\nabla u(x)|: 
||u||_{h^\infty({\mathbb B}_R)} \leq 1 \big \} \leq \lim_{|x|\rightarrow R}
\;(R-|x|){\cal K}_{n, R}( x)= C_n
\end{equation} 
hold, where 
\begin{equation} \label{Eq_2.1.9}
C_n=\frac{n-2}{2\pi}\sup _{\gamma \geq 0}\;
\frac{1}{\sqrt{1+\gamma^2}}\int _ {0}^{\pi}\sin ^{n-3}
\varphi \;d\varphi\int _ {0}^{\pi/2}
\big |{\mathcal P}_{n}(\vartheta, \varphi; \gamma)\big |
\sin^{n-2}\vartheta\;d\vartheta \;,
\end{equation}
with
\begin{eqnarray*}
{\mathcal P}_{n}(\vartheta, \varphi;  \gamma)&=&
n\gamma\sin 2\vartheta\cos \varphi+n\cos 2 \vartheta+(n-2)\\
& & \\
&=&2\big [n\gamma\cos \vartheta \sin 
\vartheta\cos \varphi+(n\cos^2 \vartheta -1) \big ].
\end{eqnarray*}

According to Proposition 1 in \cite{KM12}, the sharp coefficient 
${\mathcal C}_n(x)$ in the inequality
\begin{equation} \label{Eq_2.1.10}
|\nabla u(x)|\leq {\mathcal C}_n(x)||u||_{h^\infty({\mathbb R}^{n} _{+})} \;,
\end{equation}
where $u$ is a bounded harmonic function in the half-space ${\mathbb R}^{n}_{+}$,
is equal to ${\mathcal C}_n(x)=C_n/x_n$ with the best constant $C_n$ given by (\ref{Eq_2.1.9}). 
By Theorem 1 in \cite{KM12}, the value of $C_n$ is given by the formula
\begin{equation} \label{Eq_2.1.12}
C_n=\frac{4(n-1)^{(n-1)/2}\;\omega_{n-1}}{ n^{n/2}\;\omega_n}\;.
\end{equation}

Let $R$ denote the radius of the ball $B\subset\Omega$ tangent to $\partial\Omega$
at the point $\xi$. We put the origin ${\mathcal O}$ at the center of $B$. 
Let the point $x$ belong to the interval joining ${\mathcal O}$ and $\xi$. 
Then $R-|x|=|x-\xi|$. By (\ref{Eq_2.1.8}) with $C_n$ from (\ref{Eq_2.1.12}) 
on the right-hand side we conclude the proof in the case $n\geq 3$ by reference
to the inequality
\begin{equation} \label{Eq_2.1.13}
||u||_{h^\infty(B)}\leq \sup_\Omega |u|\;.
\end{equation}

The proof of Theorem 2 in the case $n=2$ is analogous, estimate 
(\ref{Eq_2.1.8}) follows from D. Khavinson's \cite{KHAV} inequality
\begin{equation} \label{Eq_KH}
|f'(z)|\leq {4R \over \pi (R^2-|z|^2)}\sup_{|\zeta|<R} |\Re f(\zeta)|
\end{equation}
by the change $f=u+iv$, $f'(z)=u'_x-iv'_y$, where $f$ is holomorphic in 
${\mathbb D}_R=\{ z\in {\mathbb C}: |z|<R \}$.
The estimate (\ref{Eq_2.1.10}) results from (\ref{Eq_HP})
by the change $f=u+iv$, $f'(z)=u'_x-iv'_y$, where $f$ is holomorphic in 
${\mathbb C}_+$.
\end{proof}

\begin{remark} {\bf 1}. The following inequality for the modulus of 
the gradient of a harmonic function is known (see \cite{PW}, Ch. 2, Sect. 13) 
$$
|\nabla u(x)| \leq {A_n\over d_x}\;{\rm osc}_{_{\Omega}} (u)\;,
$$
where
$$
A_n={n\omega_{n-1}\over (n-1)\omega_n}\;.
$$
It is equivalent to the estimate
\begin{equation} \label{Eq_1.2}
|\nabla u(x)| \leq {2A_n\over d_x}\;\sup_\Omega |u|\;,
\end{equation}
where $u$ is a bounded harmonic function in $\Omega\subset{\mathbb R}^n,\; n\geq 2$,
and ${\rm osc}_{_{\Omega}} (u)$ is the oscillation of $u$ on $\Omega$.

The coefficient on the right-hand side of (\ref{Eq_1.2}) is less than
that in the well known gradient estimate (see, e.g., \cite{GiTr}, Sect. 2.7)
$$
|\nabla u(x)| \leq {n\over d_x}\;\sup_\Omega |u|\;.
$$

By
$$
{C_n \over 2A_n}={2 \over \sqrt{n}}{\left ( 1 -{1\over n} \right )^{(n+1)/2}}<1\;,
$$
inequality (\ref{Eq_1.AG}) with $C_n$ from (\ref{Eq_1.AB}) improves (\ref{Eq_1.2}) for 
domains complementary to convex closed domains. 

\smallskip
Sharp estimates of derivatives of harmonic functions can be found in the books 
\cite{KM}, \cite{KM15}. We also mention the articles \cite{AA}, \cite{KAMA1}, 
\cite{KAMA3} dealing with estimates of harmonic functions.
\end{remark}

\setcounter{equation}{0}
\section{Estimates for the maximum value of the modulus of directional derivative of a
vector field with harmonic components}

Let in the domain $\Omega\subset {\mathbb R}^n$, there is a $m$-component 
vector field ${\bs a(x)}=(a_{_1}(x),\dots,a_m(x))$, $m\geq 1$. Let, further ${\bs \ell}=(\ell_1,\dots,\ell_n)$ be a unit $n$-dimensional vector. 
The derivative of the field ${\bs a(x)}$ in the direction ${\bs \ell}$ is defined by
$$
{\partial {\bs a}\over \partial \bs \ell}=\lim _{t \rightarrow 0}
{\bs a(x+t \bs \ell)-\bs a(x) \over t}\;,
$$
that is
\begin{equation} \label{Eq_3.1}
{\partial {\bs a}\over \partial \bs \ell}=(\bs \ell, \nabla ){\bs a}\;.
\end{equation}

Let us introduce some notation used in the sequel.
By $||\bs u||_{[L^\infty(\partial\Omega)]^m} =\mbox{ess}\;\sup \{ |\bs u(x) |:
x \in \partial\Omega \}$ we denote the norm in the space $[L^\infty(\partial\Omega)]^m$ of vector-valued functions $\bs u$ on $\partial \Omega$ with $m$ components from $L^\infty(\partial\Omega)$. 
By $[h^\infty(\Omega)]^m$ we mean  the Hardy space of 
vector-valued functions $\bs u(x)=(u_{_1}(x),\dots,u_m(x))$ with  bounded harmonic components
on $\Omega$ endowed with the norm $||\bs u||_{[h^\infty(\Omega)]^m} =\sup \{ |\bs u(x) |:
x \in \Omega \}$. 

It is known that any element of 
$[h^\infty({\mathbb R}^n_{+})]^m$ can be represented by the Poisson integral
\begin{equation} \label{Eq_3.2}
\bs u(x)=\frac{2}{\omega _n}\int _{\partial {\mathbb R}^n_{+}}\frac{x_n}{|y-x|^n}\;\bs u(y)dy'
\end{equation}
with boundary values in $[L^\infty({\partial {\mathbb R}^n_{+}})]^m$,
where $y=(y', 0)$, $y' \in {\mathbb R}^{n-1}$.

Now, we find a representation for the sharp coefficient ${\mathcal C}_{m,n}(x)$
in the inequality 
\begin{equation} \label{Eq_3.3}
\max _{|\bs \ell|=1}\big |(\bs \ell, \nabla ){\bs u}(x) \big |\leq{\mathcal C}_{m,n}(x)||\bs u||_{[L^\infty(\partial{\mathbb R}^n_+)]^m}\;,
\end{equation}
where $\bs u\in [h^\infty({\mathbb R}^n_{+})]^m$ and $x\in {\mathbb R}^n_{+}$.

\setcounter{theorem}{0}
\begin{lemma} \label{L_1}
Let $\bs u \in [h^\infty({\mathbb R}^n_{+})]^m$, and let  $x $ be  an arbitrary point in
${\mathbb R}^n _+$. The sharp coefficient ${\mathcal C}_{m,n}(x)$ in 
$(\ref{Eq_3.3})$ is given by
\begin{equation} \label{Eq_3.4}
{\mathcal C}_{m,n}(x)=C_{m,n}x_n^{-1}\;,
\end{equation}
where
\begin{equation} \label{Eq_3.5}
C_{m,n}=\frac{1}{\omega _n}\max _{|\bs \ell|=1}
\int _ {{\mathbb S}^{n-1}}
\big |\big (\bs e_n -n(\bs e_{\sigma}, \bs e_n)\bs e_{\sigma},\; \bs \ell \big )\big |
\;d\sigma \;,
\end{equation}
and $\bs e_\sigma $ stands for the $n$-dimensional unit vector joining the origin to 
a point $\sigma $ on the sphere ${\mathbb S}^{n-1}$.
\end{lemma}
\begin{proof} Let $x=(x', x_n)$ be a fixed point in ${\mathbb R}^n_+$.
The representation (\ref{Eq_3.2}) implies
$$
\frac{\partial \bs u}{\partial x_j}=\frac{2}{\omega _n}\int _{\partial {\mathbb R}^n_{+}}
\left [ \frac{\delta_{nj}}{|y-x|^n} + \frac{nx_n (y_j-x_j)}{|y-x|^{n+2}}  \right ]\bs u(y)dy',
$$
that is, by (\ref{Eq_3.1}),
\begin{eqnarray*}
{\partial {\bs u}\over \partial \bs \ell}
&=&\frac{2}{\omega _n}\sum_{j=1}^n \;\ell_j \int _{\partial {\mathbb R}^n_{+}}
\left [\; \frac{\delta_{nj}}{|y-x|^n} + \frac{nx_n (y_j-x_j)}{|y-x|^{n+2}}\;\right ]\bs u(y)dy'\\
& &\\
&=&\frac{2}{\omega _n}\int _{\partial {\mathbb R}^n_{+}}
\frac{(\bs e_{n}  - n(\bs e_{xy}, \bs e_{n})\bs e_{xy},\; \bs \ell)}{|y-x|^n}\; \bs u(y)dy,
\end{eqnarray*}
where $\bs e_{xy}=(y-x)|y-x|^{-1}$. For any $\bs z \in {\mathbb S}^{m-1}$,
$$
\big ((\bs \ell, \nabla ){\bs u}(x), \bs z \big )=
\frac{2}{\omega _n}\int _{\partial {\mathbb R}^n_{+}}
\frac{(\bs e_{n}  - n(\bs e_{xy}, \bs e_{n})\bs e_{xy},\; \bs \ell)}
{|y-x|^n}\;(\bs u(y), \bs z)dy'.
$$
Hence,
\begin{eqnarray*}
{\mathcal C}_{m,n}(x)&=&\frac{2}{\omega _n}\max _{|\bs \ell|=1}\int _{\partial {\mathbb R}^n_{+}}
\frac{\big |\big (\bs e_{n}  - n(\bs e_{xy}, \bs e_{n})\bs e_{xy}, \bs \ell \big )\big | }
{|y-x|^n}\;dy'\\
& &\\
&=&\frac{1}{\omega _n x_n}\max _{|\bs \ell|=1}\int _ {{\mathbb S}^{n-1}}
\big |\big (\bs e_n -n(\bs e_{\sigma}, \bs e_n)\bs e_{\sigma},\; \bs \ell \big )\big |
\;d\sigma \;.
\end{eqnarray*}
The last equality proves (\ref{Eq_3.4}) and (\ref{Eq_3.5}).
\end{proof}

By Lemma \ref{L_1}, the sharp coefficient ${\mathcal C}_{m,n}(x)$ 
in inequality (\ref{Eq_3.3}) does not depend on $m$. Thus,
${\mathcal C}_{m,n}(x)={\mathcal C}_{1,n}(x)={\mathcal C}_n(x)$, where ${\mathcal C}_n(x)=
C_n x_n^{-1}$ is the sharp coefficient in (\ref{Eq_2.1.10}).
Thus, we arrive at the following generalization of Theorem 1 in our paper 
\cite{KM12}, where the case $m=1$ is treated.

\begin{proposition} \label{P_3}
Let $\bs u \in [h^\infty({\mathbb R}^n_{+})]^m$ and let  $x $ be an arbitrary point in
${\mathbb R}^n _+$. The inequality
\begin{equation} \label{Eq_3.3A}
\max _{|\bs \ell|=1}\big |(\bs \ell, \nabla ){\bs u}(x) \big |
\leq C_n x_n^{-1}||\bs u||_{[L^\infty(\partial{\mathbb R}^n_+)]^m}
\end{equation}
holds, where the best constant $C_n$ is the same as in Theorem $1$. 
\end{proposition}

The assertion below is an extension of Theorem 1.

\setcounter{theorem}{1}
\begin{proposition} \label{P_4} 
Let $\Omega={\mathbb R}^n\backslash {\overline G}$, where $G$ is a 
convex subdomain of ${\mathbb R}^n$, and let $\bs u$ be a
vector-valued function with $m$ bounded harmonic components in $\Omega$. 
Then for any point $x\in \Omega$ the inequality
\begin{equation} \label{Eq_1.A}
\max _{|\bs \ell|=1}\big |(\bs \ell, \nabla ){\bs u}(x) \big |\leq {C_n \over d_x}\;
\sup_\Omega |\bs u|  
\end{equation}
holds, where the constant  $C_n$ is the same as in Theorem $1$. 
\end{proposition}
\begin{proof} Let $\xi \in \partial\Omega$ be the point at $\partial\Omega$ nearest to
$x \in \Omega$. Let the notation 
${\mathbb R}^n_\xi$ be the same as in the proof of Theorem 1.
By Proposition \ref{P_3},
$$
\max _{|\bs \ell|=1}\big |(\bs \ell, \nabla ){\bs u}(x) \big |\leq {C_n\over d_x}\;||\bs u||_{[h^\infty({\mathbb R}^n_\xi)]^m}\;,
$$
where $C_n$ is given by (\ref{Eq_1.AB}). Then, using the inequality
\begin{equation} \label{Eq_ine}
||\bs u||_{[h^\infty({\mathbb R}^n_\xi)]^m}\leq \sup_\Omega |\bs u|\;,
\end{equation}
we arrive at (\ref{Eq_1.A}).
\end{proof}

Any element of $[h^\infty({\mathbb B})]^m$ can be represented 
as the Poisson integral
\begin{equation} \label{Eq_3.6}
\bs u(x)={1 \over \omega _n}\int _{{\mathbb S}^{n-1}}{1-r^2 \over |y-x|^n}\bs u(y)d\sigma_y
\end{equation}
with boundary values in $[L^\infty(\partial{\mathbb B})]^m$.

In the next assertion we find a representation for the sharp coefficient ${\mathcal K}_{m,n}(x)$
in the inequality 
\begin{equation} \label{Eq_3.7}
\max _{|\bs \ell|=1}\big |(\bs \ell, \nabla ){\bs u}(x) \big |\leq{\mathcal K}_{m,n}(x)||\bs u||_
{[L^\infty(\partial{\mathbb B})]^m}\;.
\end{equation}

\setcounter{theorem}{0}
\begin{lemma} \label{L_2}
Let $\bs u \in [h^\infty({\mathbb B})]^m$, and let  $x $ be  an arbitrary point in
${\mathbb B}$. The sharp coefficient ${\mathcal K}_{m,n}(x)$ in $(\ref{Eq_3.7})$ 
is given by
\begin{equation} \label{Eq_3.8}
{\cal K}_{m,n}( x)={1 \over \omega _n}\sup_{|\bs\ell|=1}\int _{{\mathbb S}^{n-1}}
{\big |\big ( n\left ( 1-r^2 \right )(y-x)-2|y-x|^2 x, \bs\ell \big )\big | \over |y-x|^{n+2}}\;d\sigma _y\;.
\end{equation}
\end{lemma}
\begin{proof} Fix a point $x \in {\mathbb B}$. By (\ref{Eq_3.6})
$$
{\partial \bs u \over \partial x_j}={1 \over \omega _n }\int _{{\mathbb S}^{n-1}}
\left [ {-2x_{j} \over |y-x|^n} + 
{n \left ( 1-r^2 \right ) (y_j-x_j)\over |y-x|^{n+2}}\right ]\bs u(y)d\sigma _y,
$$
that is
$$
{\partial \bs u \over \partial \bs\ell}=(\bs \ell, \nabla ){\bs u}(x)=
{1 \over \omega _n}\sum_{j=1}^n \;\ell_j\int _{{\mathbb S}^{n-1}}
{n\left ( 1-r^2 \right )(y_j-x_j)-2|y-x|^2 x_j \over |y-x|^{n+2}}\;\bs u(y)d\sigma _y.
$$
For any $\bs z \in {\mathbb S}^{m-1}$ we have
$$
\big ((\bs \ell, \nabla ){\bs u}(x), \bs z \big )={1 \over \omega _n}\int _{{\mathbb S}^{n-1}}
{\big ( n\left ( 1-r^2 \right )(y-x)-2|y-x|^2 x, \bs\ell \big ) \over |y-x|^{n+2}}\;
\big (\bs u(y), \bs z \big )d\sigma _y\;,
$$
which implies (\ref{Eq_3.8}).
\end{proof}

The next assertion is a generalization of Theorem 2.

\setcounter{theorem}{2}
\begin{proposition} \label{P_5} Let $\Omega$ be a domain in ${\mathbb R}^n$.
Let $\bs{\mathfrak U}(\Omega)$ be the set of $m$-component vector-valued functions 
$\bs u$ whose components are harmonic in $\Omega$, with $\sup_\Omega |\bs u|\leq 1$.
Assume that a point $\xi \in \partial\Omega$ can be touched by an interior ball $B$. Then
\begin{equation} \label{Eq_1.5}
\limsup _{x\rightarrow \xi}\sup_{\bs u\in \bs{\mathfrak U}(\Omega)}\max_{|\bs \ell|=1}
|x-\xi|\big |(\bs \ell, \nabla ){\bs u}(x) \big | \leq C_n\;,
\end{equation}
where $x$ is a point of the radius of $B$ directed from the center to $\xi$. 
Here the constant $C_n$ is the same as in Theorem $1$.
\end{proposition}
\begin{proof}
By Lemma \ref{L_2}, ${\mathcal K}_{m,n}(x)$ does not depend on $m$ and therefore
${\mathcal K}_{m,n}(x)={\mathcal K}_{1,n}(x)={\mathcal K}_n(x)$, where ${\mathcal K}_n(x)$
is the sharp coefficient in (\ref{Eq_2.1.1}). Hence (\ref{Eq_3.7}) can be written in the form
$$
\max _{|\bs \ell|=1}\big |(\bs \ell, \nabla ){\bs u}(x) \big |\leq{\mathcal K}_n(x)||\bs u||_
{[L^\infty(\partial{\mathbb B})]^m}\;.
$$
By dilation in the last inequality we obtain the analogue of Lemma 2
\begin{equation} \label{Eq_3.9}
\max _{|\bs \ell|=1}\big |(\bs \ell, \nabla ){\bs u}(x) \big |\leq{\mathcal K}_{n, R}(x)||\bs u||_
{[L^\infty(\partial{\mathbb B}_R)]^m}\;,
\end{equation}
where $x \in {\mathbb B}_R$ and $\bs u \in [h^\infty({\mathbb B}_R)]^m$. 
Now, (\ref{Eq_3.9}) along with the representation of ${\mathcal K}_{n, R}(x)$
from Lemma 2 leads to the inequality
$$
\limsup_{|x|\rightarrow R}\;\sup \big \{ (R-|x|)\big |(\bs \ell, \nabla ){\bs u}(x) \big |:
|\bs \ell|=1,\; 
||\bs u||_{[h^\infty({\mathbb B}_R)]^m} \leq 1 \big \} \leq \lim_{|x|\rightarrow R}
\;(R-|x|){\cal K}_{n, R}( x)= C_n\;,
$$
where $C_n$ is given by (\ref{Eq_2.1.9}). The proof is completed in the same way as
that of Theorem 2, with the only difference that (\ref{Eq_2.1.13}) is replaced by the
inequality
\begin{equation} \label{Eq_est}
||\bs u||_{[h^\infty(B)]^m}\leq \sup_\Omega |\bs u|.
\end{equation}
\end{proof}

\setcounter{equation}{0}
\section{Estimates for the divergence of a vector field with harmonic components}

Let ${\bs u(x)}=(u_{_1}(x),\dots,u_n(x))$ be a vector field with $n$ 
bounded harmonic components in $\Omega\subset{\mathbb R}^n$.

\begin{proposition} \label{P_6} Let $\bs u \in [h^\infty({\mathbb R}^n_{+})]^n$, 
and let  $x $ be  an arbitrary point in ${\mathbb R}^n _+$.
The sharp coefficient ${\mathcal D}_n(x)$ in the inequality
\begin{equation} \label{Eq_5.1}
\left |{\rm div}\;\bs u (x)\right |\leq
{\mathcal D}_n(x)||\bs u||_{[L^\infty(\partial{\mathbb R}^n_+)]^n}
\end{equation}
is given by
\begin{equation} \label{Eq_5.2}
{\mathcal D}_n(x)=D_n x_n^{-1}\;,
\end{equation}
where
\begin{equation} \label{Eq_5.3}
D_n={2 \omega_{n-1} \over \omega_n}\int_0^{\pi/2}
\big [1+n(n-2)\cos^2\vartheta \big ]^{1/2}\sin^{n-2}\vartheta d\vartheta\;.
\end{equation}

In particular, 
$$
D_2=1\;,\;\;\;\;\;\;\;\;D_3=1+{\sqrt{3} \over 6}
\ln\big (2+\sqrt{3}\big )\;.
$$
\end{proposition}
\begin{proof} By (\ref{Eq_3.2}),
\begin{eqnarray} \label{Eqn_3.2}
{\rm div}\;\bs u&\!\!\!\!=\!\!\!\!&{2\over \omega _n}\sum_{j=1}^n\int _{\partial {\mathbb R}^n_{+}}u_j(y){\partial \over \partial x_j}\left ( {x_n\over |y-x|^n}\right )dy'=\nonumber\\
& &\nonumber\\
& &{2\over \omega _n}\sum_{j=1}^n\int _{\partial {\mathbb R}^n_{+}}
\left ( {\delta_{jn}\over |y-x|^n}+{nx_n(y_j-x_j)\over |y-x|^{n+2}}\right )u_j(y)dy'=\nonumber\\
& &\nonumber\\
& &{2\over \omega _n}\sum_{j=1}^n\int _{\partial {\mathbb R}^n_{+}}
\left ( {\delta_{jn}-n(\bs e_{xy}, \bs e_n)(\bs e_{xy}, \bs e_j)\over |y-x|^n}\right ) u_j(y)dy'\;,
\end{eqnarray}
which implies
\begin{equation} \label{Eq_5.4}
{\rm div}\;\bs u={2\over \omega _n}\int _{\partial {\mathbb R}^n_{+}}
 {\big (\bs e_n-n(\bs e_{xy}, \bs e_n)\bs e_{xy},\;\bs u(y)\big ) \over |y-x|^n} dy'\;.
\end{equation}
This equality shows that the sharp coefficient ${\mathcal D}_n(x)$ in (\ref{Eq_5.1})
is represented in the form
$$
{\mathcal D}_n(x)={2\over \omega _n}\int _{\partial {\mathbb R}^n_{+}}
{\big | \bs e_n-n(\bs e_{xy}, \bs e_n)\bs e_{xy} \big | \over |y-x|^n} dy'\;.
$$
Then
\begin{equation} \label{Eq_5.5}
{\mathcal D}_n(x)={2\over \omega _nx_n}\int _{\partial {\mathbb R}^n_{+}}
\big | \bs e_n-n(\bs e_{xy}, \bs e_n)\bs e_{xy} \big |{x_n \over |y-x|^n} dy'={D_n \over x_n}\;,
\end{equation}
where
\begin{equation} \label{Eq_5.6}
D_n={2\over \omega _n}\int _{{\mathbb S}^{n-1}_{-}}
\big | \bs e_n-n(\bs e_\sigma, \bs e_n)\bs e_\sigma \big |d\sigma=
{1\over \omega _n}\int _{{\mathbb S}^{n-1}}
\big | \bs e_n-n(\bs e_\sigma, \bs e_n)\bs e_\sigma \big |d\sigma\;.
\end{equation}
The identity
\begin{equation} \label{Eq_5.6A}
\big | \bs e_n-n(\bs e_\sigma, \bs e_n)\bs e_\sigma \big |^2=1+n(n-2)(\bs e_\sigma, \bs e_n)^2\;,
\end{equation}
along with (\ref{Eq_5.6}) leads to the formula
\begin{equation} \label{Eq_5.6AB}
D_n={1\over \omega _n}\int _{{\mathbb S}^{n-1}}
\Big ( 1+n(n-2)(\bs e_\sigma, \bs e_n)^2 \Big )^{1/2}d\sigma\;.
\end{equation}
Using
\begin{equation} \label{Eq_5.7B}
\int _{{\mathbb S}^{n-1}}f\big ( (\bs \xi, \bs y) \big ) d\sigma _y=
\omega _{n-1}\int_{-1}^{1}f\big ( |\bs \xi|t \big )(1-t^2)^{(n-3)/2}dt
\end{equation}
(see, e.g., \cite{PBM3}, \textbf{4.3.2(2)}) and the
change of variable $t=\cos \vartheta$, we obtain
\begin{equation} \label{Eq_5.7A}
\int _{{\mathbb S}^{n-1}}\Big ( 1+n(n-2)(\bs e_\sigma, \bs e_n)^2 \Big )^{1/2}d\sigma=
2\omega_{n-1}\int_0^{\pi/2}
\big [1+n(n-2)\cos^2\vartheta \big ]^{1/2}\sin^{n-2}\vartheta d\vartheta\;.
\end{equation}
By (\ref{Eq_5.5}), (\ref{Eq_5.6AB}) and (\ref{Eq_5.7A}),
we arrive at (\ref{Eq_5.2}) and (\ref{Eq_5.3}).
\end{proof}

The next assertion is analogous to Proposition 2. Here the divergence replaces
the directional derivative.

\begin{proposition} \label{P_7} 
Let $\Omega={\mathbb R}^n\backslash {\overline G}$, where $G$ be a 
convex subdomain of ${\mathbb R}^n$, and let $\bs u$ be a $n$-component 
vector-valued function with bounded harmonic components in $\Omega$.
Then for any point $x\in \Omega$ the inequality
\begin{equation} \label{Eq_5.7BB}
\left |{\rm div}\;\bs u (x)\right |\leq {D_n \over d_x}\;\sup_\Omega |\bs u|  
\end{equation}
holds, where the constant $D_n$ is the same as in Proposition $\ref{P_6}$.
\end{proposition}
\begin{proof} Let $\xi \in \partial\Omega$ be a point at $\partial\Omega$ nearest to
$x \in \Omega$. Let the notation 
${\mathbb R}^n_\xi$ be the same as in the proof of Theorem 1.
By Proposition \ref{P_6}, 
$$
\left |{\rm div}\;\bs u (x)\right |\leq {D_n\over d_x}\;
||\bs u||_{[h^\infty({\mathbb R}^n_\xi)]^n}\;,
$$
where $D_n$ is defined by (\ref{Eq_5.3}). Then by (\ref{Eq_ine}) with $m=n$,
we arrive at (\ref{Eq_5.7BB}).
\end{proof}

\setcounter{theorem}{0}
\begin{lemma} \label{L_1A}
Let $\bs u \in [h^\infty({\mathbb B})]^n$, and let $x $ be an arbitrary point in ${\mathbb B}$. 
The sharp coefficient ${\cal T}_n(x)$ in the inequality 
\begin{equation} \label{Eq_5.7}
|{\rm div}\;\bs u|\leq {\cal T}_n( x)||\bs u||_{[L^\infty({\partial\mathbb B})]^n}
\end{equation}
is given by
\begin{equation} \label{Eq_5.8}
{\mathcal T}_n(x)={ 2^{n-1}\omega_{n-1} \over \omega_n(1+r)^{n-1}(1-r)}
\int_0^{\pi/2} {\left [ \big (n-(n-2)r \big )^2 +4n(n-2)r\cos^2\vartheta \right ]^{1/2} 
\over \left [1+\left ({1-r \over 1+r} \right )^2 \tan^2 \vartheta \right ]^{(n-2)/2}}
\sin^{n-2}\vartheta\;d\vartheta \;.
\end{equation} 

In particular,
$$
{\mathcal T}_2(x)={2 \over 1-r^2}\;,\;\;\;\;\;\;\;\;\;\;\;
{\mathcal T}_3(x)={1 \over 1-r^2}\left (2+{3-r^2 \over 2\sqrt{3} r}\ln{\sqrt{3}+r \over
\sqrt{3} -r} \right )\;.
$$
\end{lemma}
\begin{proof} Let us fix a point $x \in {\mathbb B}$. By (\ref{Eq_3.6}) we have 
$$
{\partial u_j \over \partial x_j}={1 \over \omega _n }\int _{{\mathbb S}^{n-1}}
\left [ {-2x_{j} \over |y-x|^n} + 
{n \left ( 1-r^2 \right ) (y_j-x_j)\over |y-x|^{n+2}}\right ] u_j(y)d\sigma _y.
$$
Therefore,
$$
{\rm div}\;\bs u={1 \over \omega _n }\int _{{\mathbb S}^{n-1}}
\left ( {-2x \over |y-x|^n} + 
{n \left ( 1-r^2 \right ) (y-x)\over |y-x|^{n+2}},\;\bs u(y)\right ) d\sigma _y.
$$
This implies that the sharp coefficient ${\cal T}_n( x)$ in (\ref{Eq_5.7})
has the form
$$
{\cal T}_n( x)={1 \over \omega _n }\int _{{\mathbb S}^{n-1}}{\big |-2x |y-x|^2+ n ( 1-r^2 ) (y-x) \big |\over |y-x|^{n+2}} d\sigma _y\;,
$$
which leads to the formula
\begin{equation} \label{Eq_5.10}
{\cal T}_n( x)={1 \over \omega _n }\int _{{\mathbb S}^{n-1}}{\big (4r^2+a^2(r)-4a(r)(x, y) \big )^{1/2}\over \big ( 1-2(x, y)+r^2 \big )^{(n+1)/2}} d\sigma _y\;,
\end{equation}
where $a(r)=2r^2+n(1-r^2)$. Transforming the integral in (\ref{Eq_5.10}) with help
of (\ref{Eq_5.7B}), we obtain
$$
{\cal T}_n( x)={\omega _{n-1} \over \omega _n }\int_{-1}^{1}
{\big (4r^2+a^2(r)-4ra(r)t \big )^{1/2}\over 
\big ( 1-2rt+r^2 \big )^{(n+1)/2}}(1-t^2)^{(n-3)/2}dt \;.
$$
Changing the variable $t=\cos \theta$, we derive
\begin{equation} \label{Eq_5.11}
{\cal T}_n( x)={\omega _{n-1} \over \omega _n }\int_{0}^{\pi}
{\big (4r^2+a^2(r)-4ra(r)\cos \theta \big )^{1/2}\over 
\big ( 1-2r\cos \theta+r^2 \big )^{(n+1)/2}}\sin^{n-2} \theta d\theta \;.
\end{equation}
Finally, setting 
$$
\theta=2\arctan \left ( {1-r \over 1+r}\tan \vartheta \right )
$$ 
in (\ref{Eq_5.11}) and using (\ref{S_1})-(\ref{S_3}), we arrive at (\ref{Eq_5.8}).
\end{proof} 

By dilation in Lemma \ref{L_1A}, we obtain   

\begin{lemma} \label{L_2A}
Let $\bs u \in [h^\infty({\mathbb B}_R)]^n$, and let $x $ be an arbitrary point in 
${\mathbb B}_R$. The sharp coefficient ${\cal T}_{n, R}(x)$ in the inequality
$$
|{\rm div}\; \bs u(x)|\leq {\cal T}_{n, R}( x)||\bs u||_{[L^\infty(\partial{\mathbb B}_R)]^n}
$$
is given by
$$
{\mathcal T}_{n, R}(x)={ 2^{n-1}\omega_{n-1}R^{n-1} \over \omega_n(R+|x|)^{n-1}(R-|x|)}
\int_0^{\pi/2} Q_n \left ( \vartheta;\; {|x|\over R} \right )\sin^{n-2}\vartheta\;d\vartheta \;,
$$
where 
$$
Q_n( \vartheta; r)={\left [ \big (n-(n-2)r \big )^2 +4n(n-2)r\cos^2\vartheta \right ]^{1/2} 
\over \left [1+\left ({1-r \over 1+r} \right )^2 \tan^2 \vartheta \right ]^{(n-2)/2}}\;.
$$
\end{lemma}

\setcounter{theorem}{5}
\begin{proposition} \label{T_1AA}
Let $\Omega$ be a domain in ${\mathbb R}^n$ and
let $\bs{\mathfrak U}(\Omega)$ be the set of $n$-component vector-valued functions 
$\bs u$ whose components are harmonic in $\Omega$, and $\sup_\Omega |\bs u|\leq 1$.
Suppose that a point $\xi \in \partial\Omega$ can be touched by an interior ball $B$. Then
$$
\limsup _{x\rightarrow \xi}\sup_{\bs u\in \bs{\mathfrak U}(\Omega)}
|x-\xi|\big |{\rm div}\;\bs u(x) \big | \leq D_n\;,
$$
where $x$ is a point of the radius of $B$ directed from the center to $\xi$. 
Here the constant $D_n$ is the same as in Proposition $\ref{P_6}$.
\end{proposition}
\begin{proof}
By Lemma \ref{L_2A}, the relations
\begin{equation} \label{Eq_5.12}
\limsup_{|x|\rightarrow R}\;\sup \big \{ (R-|x|)|{\rm div}\;\bs u(x)\big |: 
||\bs u||_{[h^\infty({\mathbb B}_R)]^n} \leq 1 \big \} \leq \lim_{|x|\rightarrow R}
\;(R-|x|){\cal T}_{n, R}( x)= D_n
\end{equation} 
hold, where $D_n$ is the same as in Proposition $\ref{P_6}$.

Using the notation introduced in Theorem 2, by (\ref{Eq_5.12}) and (\ref{Eq_est}) with 
$m=n$ the result follows.
\end{proof}

\setcounter{equation}{0}
\section{Estimates for the divergence of an elastic displacement field and 
the pressure in a fluid}

Let $[{\rm C}_{\rm b}(\partial{\mathbb R}^n_+)]^n$ 
be the space of vector-valued functions
with $n$ components which are bounded and continuous on $\partial{\mathbb R}^n_+$.
This space is endowed with the norm $||\bs u||_{[{\rm C}_{\rm b}(\partial{\mathbb R}^n_+)]^n}
=\sup \{|\bs u(x)|: x \in \partial{\mathbb R}^n_+ \}$.

In the half-space ${\mathbb R}^n _{+},\; n\geq 2,$ consider the Lam\'e system
\begin{equation} \label{Eq_4.1}
\Delta \bs u +(1-2\sigma)^{-1}\;{\rm grad}\;{\rm div}\; \bs u=\bs 0\;,
\end{equation}
and the Stokes system
\begin{equation}  \label{Eq_4.2}
\Delta \bs u -{\rm grad}\; p=\bs 0\;,\;\;\;\;\;\;\;{\rm div}\;\bs u=0\;,
\end{equation}
with the boundary condition
\begin{equation} \label{Eq_4.3}
\bs u\big |_{x_{n}=0}=\bs f,
\end{equation}
where $\sigma$ is the Poisson coefficient,
$\bs f \in [{\rm C}_{\rm b}(\partial {\mathbb R}^n_{+})]^n$, 
$\bs u=(u_1,\dots ,u_n)$ is the displacement vector of an elastic medium or the 
velocity vector of a fluid, and $p(x)$ is the pressure in the fluid vanishing as
$x_n\rightarrow \infty$. 

We assume that $ \sigma \in (-\infty, 1/2)\cup(1, +\infty)$ which means the strong 
ellipticity of system (\ref{Eq_4.1}). By $\lambda$ and $\mu$ we denote the 
Lam\'e constants. Since $\sigma=\lambda /2(\lambda+\mu)$ 
the strong ellipticity is equivalent to the inequalities
$\mu>0, \lambda+\mu>0$ and $-\mu<\lambda+\mu<0$. 

\smallskip
A unique solution  $\bs u\in [{\rm C}^{2}({\mathbb R}^n _{+})]^n \cap [{\rm C}_{\rm b}
(\overline {{\mathbb R}^n _{+}})]^n$ of problem (\ref{Eq_4.1}), (\ref{Eq_4.3}) 
and the vector component $\bs u\in [{\rm C}^{2}({\mathbb R}^n _{+})]^n 
\cap [{\rm C}_{\rm b}(\overline {{\mathbb R}^n _{+}})]^n$
of a solution $\{ \bs u, p \}$ to problem (\ref{Eq_4.2}), (\ref{Eq_4.3}) 
admit the representation (see, e.g., \cite{KM15}, pp. 64-65)
\begin{equation} \label{Eq_4.4}
\bs u(x)=\int _{\partial {\mathbb R}^n_{+}}{\cal H}\left ({y-x\over |y-x|}\right )
{x_n\over |y-x|^n}\bs f(y')dy',
\end{equation}
where $x\in {\mathbb R}^n _{+}$, $y=(y', 0)$, $y' \in {\mathbb R}^{n-1}$.
Here ${\mathcal H}$ is the $(n\times n)$-matrix-valued function on ${\mathbb S}^{n-1}$ 
with elements
\begin{equation} \label{Eq_4.5}
{2\over \omega _n}\left ( (1-\kappa )\delta _{jk}+n\kappa
{(y_{j}-x_{j})(y_{k}-x_{k})\over |y-x|^2}\right ),
\end{equation}
where $\kappa =1$ for the Stokes system and $\kappa =(3-4\sigma)^{-1}$
for the Lam\'e system.

\begin{proposition} \label{P_9} {\rm (i)} Let $\bs u\in [{\rm C}^{2}({\mathbb R}^n _{+})]^n 
\cap [{\rm C}_{\rm b}(\overline {{\mathbb R}^n _{+}})]^n$ be a solution 
of the Lam\'e system in ${\mathbb R}^n _{+}$.
The sharp coefficient ${\mathcal E}_n(x)$ in the inequality
\begin{equation} \label{Eq_4.7}
\left |{\rm div}\;\bs u (x)\right |\leq
{\mathcal E}_n(x)||\bs u||_{[{\rm C}_{\rm b}(\partial{\mathbb R}^n_+)]^n}
\end{equation}
is given by
\begin{equation} \label{Eq_4.8}
{\mathcal E}_n(x)={1-2\sigma\over 3-4\sigma}E_n x_n^{-1}\;,
\end{equation}
where
\begin{equation} \label{Eq_4.9}
E_n={4 \omega_{n-1} \over \omega_n}\int_0^{\pi/2}
\big [1+n(n-2)\cos^2\vartheta \big ]^{1/2}\sin^{n-2}\vartheta d\vartheta\;.
\end{equation}

In particular, 
$$
E_2=2\;,\;\;\;\;\;\;\;\;E_3=2\left (1+{\sqrt{3} \over 6}
\ln\big ( 2+\sqrt{3}\big ) \right )\;.
$$

{\rm (ii)} Let $\bs u\in [{\rm C}^{2}({\mathbb R}^n _{+})]^n 
\cap [{\rm C}_{\rm b}(\overline {{\mathbb R}^n _{+}})]^n$ be the vector component of a
solution $\{ \bs u, p \}$ of the Stokes system $(\ref{Eq_4.2})$ in ${\mathbb R}^n _{+}$ 
and $p(x)$ be the pressure vanishing as $x_n\rightarrow \infty$. 
The sharp coefficient ${\mathcal S}_n(x)$ in the inequality
\begin{equation} \label{Eq_4.14}
\left |p(x)\right |\leq
{\mathcal S}_n(x)||\bs u||_{[{\rm C}_{\rm b}(\partial{\mathbb R}^n_+)]^n}
\end{equation}
is given by
\begin{equation} \label{Eq_4.15}
{\mathcal S}_n(x)=E_n x_n^{-1}\;,
\end{equation}
where the constant $E_n$ is defined by $(\ref{Eq_4.9})$.
\end{proposition}
\begin{proof} {\rm (i)} {\it Proof of inequality $(\ref{Eq_4.7})$.} 
By (\ref{Eq_4.4}) and (\ref{Eq_4.5}),
\begin{equation} \label{Eq_4.10}
u_j(x)={2\over \omega _n}\int _{\partial {\mathbb R}^n_{+}}\left ( 
(1-\kappa)\bs e_j+n\kappa{(y_j-x_j)(y-x) \over |y-x|^2},\;\bs f(y')\right ){x_n\over |y-x|^n}dy'\;.
\end{equation}
Noting that $y_n=0$ in (\ref{Eq_4.10}), we find
\begin{eqnarray*}
& &\sum_{j=1}^n{\partial \over \partial x_j}\left \{ { (y_j-x_j)\big (y-x,\;
\bs f(y')\big )x_n \over |y-x|^{n+2}} \right \}=\sum_{j=1}^n{(n+2) (y_j-x_j)^2\big (y-x,\;\bs f(y')\big )x_n \over |y-x|^{n+4}}+\\
& &\\
& &\sum_{j=1}^n { -\big (y-x,\;\bs f(y')\big )x_n +(y_j-x_j)f(y')x_n+
(y_j-x_j)\big (y-x,\;\bs f(y')\big )\delta_{nj}
\over |y-x|^{n+2}}=\\
& &\\
& &{ -n\big (y-x,\;\bs f(y')\big )x_n -\big (y-x,\;\bs f(y')\big )x_n+
(y_n-x_n)\big (y-x,\;\bs f(y')\big )+(n+2)\big (y-x,\;\bs f(y')\big )\over |y-x|^{n+2}} =0.
\end{eqnarray*}
This together with (\ref{Eq_4.10}) gives
$$
{\rm div}\;\bs u(x)={2\over \omega _n}(1-\kappa)\sum_{j=1}^n\int _{\partial {\mathbb R}^n_{+}}f_j(y'){\partial \over \partial x_j}\left ( {x_n\over |y-x|^n}\right )dy'\;.
$$
Hence using (\ref{Eqn_3.2}), (\ref{Eq_5.4}) and $\kappa =(3-4\sigma)^{-1}$, we have
\begin{equation} \label{Eq_4.11A}
{\rm div}\;\bs u(x)={4(1-2\sigma)\over \omega _n(3-4\sigma)}\int _{\partial {\mathbb R}^n_{+}}
 {\big (\bs e_n-n(\bs e_{xy}, \bs e_n)\bs e_{xy},\;\bs f(y')\big ) \over |y-x|^n} dy'\;.
\end{equation}
Therefore the sharp coefficient ${\mathcal E}_n(x)$ in  (\ref{Eq_4.7})
is represented in the form
$$
{\mathcal E}_n(x)={4(1-2\sigma)\over \omega _n(3-4\sigma)}\int _{\partial {\mathbb R}^n_{+}}
{\big | \bs e_n-n(\bs e_{xy}, \bs e_n)\bs e_{xy} \big | \over |y-x|^n} dy'\;.
$$
Thus,  
\begin{equation} \label{Eq_4.11}
{\mathcal E}_n(x)={4(1-2\sigma)\over \omega _n(3-4\sigma)x_n}
\int _{\partial {\mathbb R}^n_{+}}\big | \bs e_n-n(\bs e_{xy}, \bs e_n)\bs e_{xy} \big |
{x_n \over |y-x|^n} dy'={(1-2\sigma)E_n \over (3-4\sigma)x_n}\;,
\end{equation}
where
$$
E_n={4\over \omega _n}\int _{{\mathbb S}^{n-1}_{-}}
\big | \bs e_n-n(\bs e_\sigma, \bs e_n)\bs e_\sigma \big |d\sigma=
{2\over \omega _n}\int _{{\mathbb S}^{n-1}}
\big | \bs e_n-n(\bs e_\sigma, \bs e_n)\bs e_\sigma \big |d\sigma\;.
$$
Using (\ref{Eq_5.6A}), we write the last equality as
\begin{equation} \label{Eq_4.13}
E_n={2\over \omega _n}\int _{{\mathbb S}^{n-1}}
\Big ( 1+n(n-2)(\bs e_\sigma, \bs e_n)^2 \Big )^{1/2}d\sigma\;.
\end{equation}
By (\ref{Eq_4.11}), (\ref{Eq_4.13}) and (\ref{Eq_5.7A}),
we arrive at (\ref{Eq_4.8}) and (\ref{Eq_4.9}).

\smallskip
{\rm (ii)} {\it Proof of inequality $(\ref{Eq_4.14})$.} We write (\ref{Eq_4.1}) as
\begin{equation} \label{Eq_4.17}
\Delta \bs u -{\rm grad}\;p=\bs 0\;,\;\;\;\;\;\;\;p=-\;{1\over 1-2\sigma}\;{\rm div}\;\bs u\;.
\end{equation}
It follows from (\ref{Eq_4.11A}) that ${\rm div}\;\bs u(x)\rightarrow 0$ for every 
$x\in {\mathbb R}^n_+$ as $\sigma\rightarrow 1/2$. We also see that
$$
p(x)= -\;{1\over 1-2\sigma}\;{\rm div}\;\bs u(x)=
-\;{4\over \omega _n(3-4\sigma)}\int _{\partial {\mathbb R}^n_{+}}
 {\big (\bs e_n-n(\bs e_{xy}, \bs e_n)\bs e_{xy},\;\bs f(y')\big ) \over |y-x|^n} dy'
$$
tends to
$$
-\;{4\over \omega _n} \int _{\partial {\mathbb R}^n_{+}}
{\big (\bs e_n-n(\bs e_{xy}, \bs e_n)\bs e_{xy},\;\bs f(y')\big ) \over |y-x|^n} dy'\;.
$$
as $\sigma\rightarrow 1/2$. Hence
$$
p(x)=-\;{4\over \omega _n} \int _{\partial {\mathbb R}^n_{+}}
{\big (\bs e_n-n(\bs e_{xy}, \bs e_n)\bs e_{xy},\;\bs f(y')\big ) \over |y-x|^n} dy'\;.
$$
Replacing ${\rm div}\;\bs u(x)$ by $(2\sigma -1)p(x)$ in (\ref{Eq_4.7}),
and taking the limit as $\sigma \rightarrow 1/2$, we arrive at (\ref{Eq_4.14})
with the sharp coefficient (\ref{Eq_4.15}).
\end{proof}

\smallskip
By Proposition \ref{P_9} with the same argument as in Proposition 5,  we derive

\begin{corollary} \label{C_7} Let $\Omega={\mathbb R}^n\backslash {\overline G}$, 
where $G$ is a convex domain in ${\mathbb R}^n$. Let $\bs u\in [{\rm C}^{2}(\Omega)]^n 
\cap [{\rm C}_{\rm b}(\overline {\Omega})]^n$ be a solution of the
Lam\'e system in $\Omega$. Then for any point $x\in \Omega$ the inequality
$$
\left |{\rm div}\;\bs u (x)\right |\leq 
{(1-2\sigma)E_n \over (3-4\sigma)d_x}\;\sup_\Omega |\bs u|  
$$
holds, where the constant $E_n$ is the same as in Proposition $\ref{P_9}$. 
\end{corollary} 

\begin{corollary} \label{C_8} Let $\Omega={\mathbb R}^n\backslash {\overline G}$, 
where $G$ is a convex domain in ${\mathbb R}^n$. Let $\bs u\in [{\rm C}^{2}(\Omega)]^n 
\cap [{\rm C}_{\rm b}(\overline {\Omega})]^n$ be the vector component of a
solution $\{ \bs u, p \}$ of the Stokes system $(\ref{Eq_4.2})$ in $\Omega$ and 
let $p(x)$ be the pressure vanishing as $d_x\rightarrow \infty$. 
Then for any point $x\in \Omega$ the inequality
$$
|p(x)|\leq {E_n \over d_x}\;\sup_\Omega |\bs u|  
$$
holds, where the constant $E_n$ is the same as above. 
\end{corollary}

\setcounter{equation}{0}
\section{Real-part estimates for derivatives of analytic functions}

\setcounter{theorem}{2}
\begin{theorem} \label{T_1C} 
Let $\Omega={\mathbb C}\backslash {\overline G}$, where $G$ is a convex domain in ${\mathbb C}$,
and let $f$ be a holomorphic function in $\Omega$ with bounded real part.
Then for any point $z\in \Omega$ the inequality
\begin{equation} \label{Eq_6.1}
\big |f^{(s)}(z)\big |\leq {K_s \over d_z^s}\sup_\Omega |\Re f|\;,\;\;\;\;\;\;\;\;s=1,2,\dots\;, 
\end{equation}
holds with $d_z={\rm dist}\;(z, \partial \Omega)$, where
\begin{equation} \label{Eq_6.2}
K_s={s! \over \pi}\max_{\alpha}\int_{-\pi/2}^{\pi/2} \big |\cos \big (\alpha +(s+1)
\varphi \big ) \big |\cos ^{s-1}\varphi\;d\varphi
\end{equation}
is the best constant in the inequality
\begin{equation} \label{Eq_6.3}
|f^{(s)}(z)| \leq {K_s \over (\Im z)^s }\;||\Re f||_{L^\infty(\partial{\mathbb C}_+)}
\end{equation}
for holomorphic functions $f$ in the half-plane ${\mathbb C}_+=\{ z \in {\mathbb C}: \Im z >0 \}$
with bounded real part.

In particular,
\begin{equation} \label{Eq_6.4}
K_{2l+1}={2[(2l+1)!!]^2 \over \pi(2l+1)}\;,
\end{equation}
and
\begin{equation} \label{Eq_6.5}
\hspace{-18mm}K_2={3\sqrt{3} \over 2\pi }\;,
\end{equation}
\begin{equation} \label{Eq_6.6}
\hspace{-4mm}K_4={3 (16+5\sqrt{5})\over 4\pi}\;.
\end{equation}
\end{theorem}
\begin{proof} Inequality (\ref{Eq_6.3}) with the best constant (\ref{Eq_6.2}) can
be found in \cite{KM14}. 
Let $\zeta \in \partial\Omega$ be the point nearest to $z \in \Omega$ and let 
$T(\zeta)$ be the line containing $\zeta$ and orthogonal to the line passing through 
$z$ and $\zeta$. By ${\mathbb C}_\zeta$ we denote the half-plane with the boundary 
$T(\zeta)$ which is contained in $\Omega$. Then by (\ref{Eq_6.3}),
\begin{equation} \label{Eq_7}
\big |f^{(s)}(z)\big |\leq {K_s \over d_z^s}||\Re f||_{h^\infty({\mathbb C}_\zeta)}\;,
\end{equation}
where $K_s$ is given by (\ref{Eq_6.2}). Using
$$
||\Re f||_{h^\infty({\mathbb C}_\zeta)}\leq \sup_\Omega |\Re f|\;,
$$
we obtain (\ref{Eq_6.1}).
\end{proof}

\begin{theorem} \label{T_2C} Let $\Omega$ be a domain in ${\mathbb C}$,
and let ${\mathfrak R}(\Omega)$ be the set of holomorphic functions $f$ in $\Omega$
with $\sup_\Omega|\Re f|\leq 1$.
Assume that a point $\zeta \in \partial\Omega$ can be touched by an interior disk $D$. Then
$$
\limsup _{z\rightarrow \zeta} \sup_{f\in {\mathfrak R}(\Omega)}|z-\zeta|^s\big |f^{(s)}(z)\big | \leq K_s\;,\;\;\;\;\;\;\;\;s=1,2,\dots\;, 
$$
where $z$ is a point of the radius of $D$ directed from the center to $\zeta$.
Here the constant $K_s$ is the same as in Theorem $3$ and cannot be diminished.
\end{theorem}
\begin{proof} In Theorem 7.1 of paper \cite{KM14} (see also Corollary 1 in \cite{Kr1})
the limit relation was proved:
\begin{equation} \label{Eq_6.9}
\lim_{r\rightarrow R}(R-r)^s {\mathcal H}_s(z)=K_s,
\end{equation}
where $r=|z|$, $K_s$ is the best constant (\ref{Eq_6.2}) in inequality (\ref{Eq_6.3}), 
and ${\mathcal H}_s(z)$ is the sharp coefficient in the inequality 
\begin{equation} \label{Eq_6.10}
|f^{(s)}(z)|\leq {\mathcal H}_s(z)||\Re f||_{L^\infty(\partial{\mathbb D}_R)}.
\end{equation}
Here $f$ is an analytic function with bounded real part in the disk
${\mathbb D}_R=\{ z\in {\mathbb C}: |z|<R \}$. 

Therefore, by (\ref{Eq_6.9}) and (\ref{Eq_6.10}), the relations
\begin{equation} \label{Eq_6.11}
\limsup_{r\rightarrow R}\;\sup \big \{ (R-r)^s|f^{(s)}(z)|: 
||\Re f||_{h^\infty({\mathbb D}_R)} \leq 1 \big \} \leq \lim_{|z|\rightarrow R}
\;(R-r|)^s{\mathcal H}_s(z)= K_s
\end{equation} 
hold.

Let $R$ be the radius of the interior disk $D$ tangent to $\partial\Omega$
at a point $\zeta$. We place the origin ${\mathcal O}$ at the
center of $D$. Let $z$ belong to the interval connecting
${\mathcal O}$ and $\zeta$. Then $R-r=|z-\zeta|$.
By (\ref{Eq_6.11}) and the inequality
$$
||\Re f||_{h^\infty({\mathbb D}_R)}\leq \sup_\Omega |\Re f|\;,
$$
the result follows.
\end{proof}

\begin{remark} ${\bs 2}$. We note that the estimate 
$$
\big |f^{(s)}(z)\big |\leq {4s! \over \pi d_z^s}
\sup_\Omega |\Re f|\;,\;\;\;\;\;\;\;\;s=1,2,\dots\;, 
$$
with a rougher constant than in (\ref{Eq_6.1}), holds for an arbitrary domain
$\Omega\subset {\mathbb C}$. The estimate follows from the sharp inequality
$$
\big |f^{(s)}(0)\big |\leq {4s! \over \pi R^s}\sup_{|\zeta|<R} |\Re f(\zeta)| 
$$
obtained in \cite{KM}, Section 5.3. Certain estimates for $\big |f^{(s)}(z)\big |$
in an arbitrary complex domain are obtained in \cite{KAEL}.
\end{remark}
\bigskip
{\bf Acknowledgments.}
The research of the first author was supported by the KAMEA program of 
the Ministry of Absorption, State of Israel, and by the Ariel University.

\end{document}